\pdfoutput=1
\RequirePackage{ifpdf}
\ifpdf 
\documentclass[pdftex]{sigma}
\else
\documentclass{sigma}
\fi

\numberwithin{equation}{section}

\newtheorem{Theorem}{Theorem}[section]
\newtheorem*{Theorem*}{Theorem}

\newtheorem{Lemma}[Theorem]{Lemma}
\newtheorem{Proposition}[Theorem]{Proposition}
{ \theoremstyle{definition}
\newtheorem{Definition}[Theorem]{Definition}

\newtheorem{Example}[Theorem]{Example}
\newtheorem{Remark}[Theorem]{Remark}
\newtheorem{Question}[Theorem]{Question}
}

\theoremstyle{plain}
\newtheorem*{namedthm}{\namedthmname}

\newtheorem{bigthm}{Theorem}

 \newcommand{\R}{\mathbb R}
 
 \newcommand{\C}{\mathbb C}
 
 \newcommand{\N}{\mathbb N}

 \newcommand{\vep}{\varepsilon}

 \newcommand{\f}{\varphi}

 \newcommand \psh {{\rm PSH}}
 \newcommand \PSH {{\rm PSH}}
 \newcommand \SH {{\rm SH}}

 \newcommand{\id}{{\bf 1}}
 
 \newcommand \Amp {{\rm Amp}}
 \newcommand \vol{{\rm Vol}}
 \newcommand \Ent{{\rm Ent}}
\newcommand{\Ric}{{\rm Ric}}
 \newcommand \cE{\mathcal E }

 \newcommand \cM {\mathcal{M}}

\begin{document}

\allowdisplaybreaks

\renewcommand{\thefootnote}{}

\newcommand{\arXivNumber}{2310.10152}

\renewcommand{\PaperNumber}{039}

\FirstPageHeading

\ShortArticleName{Entropy for Monge--Amp\`ere Measures in the Prescribed Singularities Setting}

\ArticleName{Entropy for Monge--Amp\`ere Measures\\ in the Prescribed Singularities Setting\footnote{This paper is a~contribution to the Special Issue on Differential Geometry Inspired by Mathematical Physics in honor of Jean-Pierre Bourguignon for his 75th birthday. The~full collection is available at \href{https://www.emis.de/journals/SIGMA/Bourguignon.html}{https://www.emis.de/journals/SIGMA/Bourguignon.html}}}

\Author{Eleonora DI NEZZA~$^{\rm a}$, Stefano TRAPANI~$^{\rm b}$ and Antonio TRUSIANI~$^{\rm c}$}

\AuthorNameForHeading{E.~Di Nezza, S.~Trapani and A.~Trusiani}

\Address{$^{\rm a)}$~IMJ-PRG, Sorbonne Universit\'e \& DMA, \'Ecole Normale Sup\'erieure, Universit\'e PSL, CNRS,\\
\hphantom{$^{\rm a)}$}~4 place Jussieu \& 45 Rue d'Ulm, 75005~Paris, France}
\EmailD{\href{mailto:eleonora.dinezza@imj-prg.fr}{eleonora.dinezza@imj-prg.fr}, \href{mailto:edinezza@dma.ens.fr}{edinezza@dma.ens.fr}}

\Address{$^{\rm b)}$~Universit\`a di Roma Tor Vergata, Via della Ricerca Scientifica 1, 00133~Roma, Italy}
\EmailD{\href{mailto:trapani@mat.uniroma2.it}{trapani@mat.uniroma2.it}}

\Address{$^{\rm c)}$~Chalmers University of Technology, Chalmers tv\"argata 3, 41296 G\"oteborg, Sweden}
\EmailD{\href{mailto:trusiani@chalmers.se}{trusiani@chalmers.se}}

\ArticleDates{Received October 16, 2023, in final form May 04, 2024; Published online May 08, 2024}

\Abstract{In this note, we generalize the notion of entropy for potentials in a relative full Monge--Amp\`ere mass $\mathcal{E}(X, \theta, \phi)$, for a model potential $\phi$. We then investigate stability properties of this condition with respect to blow-ups and perturbation of the cohomology class. We also prove a Moser--Trudinger type inequality with general weight and we show that functions with finite entropy lie in a relative energy class $\mathcal{E}^{\frac{n}{n-1}}(X, \theta, \phi)$ (provided $n>1$), while they have the same singularities of $\phi$ when $n=1$.}


\Classification{32W20; 32U05; 32Q15; 35A23}

\begin{flushright}
\begin{minipage}{65mm}
\it Dedicated to Jean-Pierre Bourguignon\\
 on the occasion of his 75th birthday
\end{minipage}
\end{flushright}

\renewcommand{\thefootnote}{\arabic{footnote}}
\setcounter{footnote}{0}

\section{Introduction}

 Probability measures with finite entropy play an important role in the search of canonical metrics in K\"ahler geometry (see, e.g., \cite{BBEGZ19,BBGZ13, BBJ15, BDL17,BEGZ10, ChCh1, ChCh2}).

Let $(X, \omega)$ be a compact K\"ahler manifold of complex dimension $n\geq 1$ and assume $\omega$ is normalized such that $\vol(\omega):=\int_X \omega^n=1$. It is well known that K\"ahler metrics with constant scalar curvature are critical points of the Mabuchi K-energy defined as
\[
\cM_{\omega}(u) := \bar{S}_{\omega} E(u) - n E_{\Ric(\omega)}(u) + \Ent(\omega_u^n, \omega^n), \qquad u \in \cE^1(X,\omega),
\]
where the first two terms are energy terms while the latter is the entropy of the Monge--Amp\`ere measure $\omega_u^n:= (\omega+ {\rm dd}^c u)^n$.

 Here, given two positive Radon measures $\mu$, $\nu$, the relative entropy~$\Ent(\mu, \nu)$ is defined as
\[
\Ent(\mu, \nu) := \int_X \log\left (\frac{{\rm d}\mu}{{\rm d}\nu}\right) {\rm d}\mu,
\]
if $\mu$ is absolutely continuous with respect to $\nu$, and $+\infty$ otherwise.

The breakthrough result of Cheng and Chen \cite{ChCh1, ChCh2} ensures that the existence of a cscK metric is equivalent to the properness of the Mabuchi functional. It is then crucial to relate the two notions of energy and entropy and to investigate them.

The first step in this direction has been done in \cite{BBEGZ19}, where the authors proved that full mass $\omega$-psh functions whose Monge--Amp\`ere measure has finite entropy have finite energy as well. In other words, they prove the following inclusion:
\[
\Ent(X,\omega)\cap \mathcal{E}(X, \omega) \subset {\mathcal E}^1(X,\omega),
\]
 where $\Ent(X,\omega)\cap \mathcal{E}(X, \omega)$ is the set of $\omega$-psh functions $u$ such that $\omega_u^n$ has finite entropy with respect to a fixed volume form, say $\omega^n$.
However, as all computable examples suggest, $\Ent(X,\omega)$ is actually contained in a higher energy class ${\mathcal E}^p(X,\omega)$ for some $p>1$ depending on the dimension.
In \cite{DGL20}, the authors indeed proved that
\[
\Ent(X,\omega)\cap \mathcal{E}(X, \omega) \subset {\mathcal E}^{\frac{n}{n-1}}(X,\omega).
\]

In a series of papers \cite{DDL3, DDL2, DDL1}, Darvas, Di Nezza and Lu developed the pluripotential theory in the relative big setting (see also \cite{Trus19, Trus20} for the K\"ahler case). This proved to be very fruitful, and it allows us to work with potentials with not necessarily full mass in the general setting of a~big cohomology class. In particular, given a big class $\{\theta\}$ and a ``special'' $\theta$-psh function~$\phi$ (called \emph{model potential}), they defined and studied relative Monge--Amp\'ere energy classes $\mathcal{E}(X, \theta, \phi)$ and~$\mathcal{E}_\chi(X, \theta, \phi)$, where $\chi\colon \R^+ \rightarrow \R^+$ is a weight function. These classes comprehend the energy classes defined in \cite{BEGZ10,GZ07} for a particular choice of $\phi$. Let us emphasize that such classes are at the heart of the variational approach for the search of (singular) K\"ahler--Einstein metrics. In the relative setting, the classes $\mathcal{E}(X, \theta, \phi)$ and $\mathcal{E}_\chi(X, \theta, \phi)$ are in turn crucial for the construction of K\"ahler--Einstein metrics with prescribed singularities, where the word ``prescribed singularities" mean that the singularities of the potential of the metric are modeled by $\phi$, as showed in \cite{DDL2, Tru1, Trus20b}.

The study of such classes, and their interplay with the (generalized) entropy is then useful to pursue the study of singular cscK metrics.

In this note, we define and study the entropy for Monge--Amp\`ere measures $\theta_\varphi^n:=(\theta+{\rm dd}^c \varphi )^n$ not necessarily with full mass, i.e.,
\[\int_X \theta_\varphi^n :=m \in (0, \vol(\theta)].\]
The function $\varphi$ belongs to a relative full mass class $ \mathcal{E}(X, \theta,\phi)$, where $\phi=P_\theta[\varphi]$ is a model potential with mass $m$. We refer to Section \ref{sec:prelim} for the definitions of all these notions.

We then define the $\theta$\emph{-entropy} of $\varphi$ as
\begin{equation*}
 {\rm Ent}_\theta(\varphi):= {\rm Ent}\big( m^{-1}\theta_\varphi^n, {\omega^n}\big).
 \end{equation*}

 We establish the stability properties of entropy with respect to proper bimeromorphic maps and we generalize the result in \cite{DGL20} relating finite entropy to finite relative energy. More precisely, given $\phi$ a model potential with positive mass, i.e., $\int_X \theta_\phi^n>0$, we prove:

\begin{bigthm}[Theorem \ref{thm: Entropy implies rel Ep}]
 With assumptions as above, we have
 \[\Ent(X,\theta)\cap \mathcal{E}(X,\theta,\phi) \subset \mathcal{E}^{\frac{n}{n-1}}(X,\theta, \phi) \]
 for any $n>1$. For $n=1$, any $\varphi\in \Ent(X,\theta)\cap \mathcal{E}(X,\theta,\phi)$ satisfies $\varphi-\phi$ bounded.
\end{bigthm}

 Along the way, we obtain a Moser--Trudinger type inequality with a general \emph{weight}, i.e., a~continuous strictly increasing function $\chi\colon [0,+\infty)\to [0,+\infty)$ such that $\chi(0)=0$ and ${\chi(+\infty)=+\infty}$.

\begin{bigthm}[Theorem \ref{thm: rel Moser--Trudinger inequality weight}]
Let $\chi_1\colon\R^+\rightarrow \R^+ $ be a weight. Let $\chi_2(t) := \int_0^t \chi_1(s)^{\frac{1}{n}} {\rm d}s$. Then there exist $c>0$, $C>0$ depending on $X$, $\theta$, $n$, $\int_X \theta_\phi^n$ and $\omega$ such that, for all $\varphi\in \mathcal{E}_{\chi_1}(X,\theta, \phi)$ with $\sup_X \varphi=-1$, we have
\[
 	\int_X \exp \big( c E_{\chi_1}(\varphi, \phi)^{-\frac{1}{n}} \chi_2(\phi-\varphi)\big)\omega^n \leq C.
\]
\end{bigthm}

The above estimate generalizes and unifies several important results present in the literature, such as \cite[Theorem 1.1]{BB11}, \cite[Theorem 2.1]{DGL20}, \cite[Theorem~2.11]{DNL21}.

 We end this note with a very natural question about stability of finiteness of the entropy under deformation of the cohomology class. More precisely, we ask the following:

\begin{Question} Let $\varphi \in \Ent(X, \theta)$. Is true that $\varphi \in \Ent(X, \theta+\varepsilon \omega)$ for $\varepsilon>0$?
\end{Question}

We emphasize that this kind of result does not seem to be known even when $\{\theta\}$ is a K\"ahler class: the (subtle) problem is that it is unknown if the absolute continuity of $\theta_\varphi^n $ with respect to~$\omega^n$ implies that $(\theta+\varepsilon \omega+{\rm dd}^c \varphi)^n $ has a density as well.

We nevertheless prove that if we get a positive answer for some $\varepsilon_0>0$, then $\varphi \in \Ent(X, \theta+t\omega)$ for all $t>0$ (see Proposition \ref{lemma: entropy all time}).

\section{Preliminaries}\label{sec:prelim}

We recall results from (relative) pluripotential theory of big cohomology classes. We borrow notation and terminology from \cite{DDL6}.

Let $(X,\omega)$ be a compact K\"ahler manifold of dimension $n$ and let $\theta$ be a smooth closed $(1,1)$-form on $X$. A function $\varphi\colon X\rightarrow \mathbb{R}\cup\{-\infty\}$ is quasi-plurisubharmonic (qpsh) if it can be locally written as the sum of a plurisubharmonic function and a smooth function, and $\varphi$ is called $\theta$-plurisubharmonic ($\theta$\emph{-psh}) if it is qpsh and $\theta+{\rm dd}^c \varphi\geq 0$ in the sense of currents. Here, ${\rm d}$ and ${\rm d}^c$ are real differential operators defined as ${\rm d}:=\partial +\bar{\partial}$, ${\rm d}^c:=\frac{\rm i}{2\pi}\left(\bar{\partial}-\partial \right)$. We let $\psh(X,\theta)$ denote the set of $\theta$-psh functions that are not identically $-\infty$.
We also assume that $\{\theta\}$ is \emph{big}, i.e., there exists $\psi \in {\rm PSH}(X,\theta)$ such that $\theta +{\rm dd}^c \psi \geq \vep \omega$ for some small constant $\vep>0$. The current $T=\theta+{\rm dd}^c\psi$ is called K\"ahler current.

We say that a $\theta$-psh function $\varphi$ has \emph{analytic singularities} if there exists a constant $c>0$ such that locally on $X$,
\[
\varphi={c}\log\sum_{j=1}^{N}|f_j|^2+g,
\]
where $g$ is bounded and $f_1,\dots,f_N$ are local holomorphic functions.
We say that $\varphi$ has analytic singularities with \emph{smooth remainder} if moreover $g$ is smooth.

The \emph{ample locus} $\Amp(\theta)$ of $\{\theta\}$ is the complement of the \emph{non-K\"ahler locus}
\[E_{\rm nK}(\theta)= \bigcap_{T \ {\text{K\"ahler current}}} \{x\in X \colon \nu(T,x)>0 \},\]
where $\nu(T,x)$ is the Lelong number of $T$ at the point $x$.
The ample locus $\Amp(\theta)$ is a Zariski open subset, and it is nonempty \cite{Bou04}. Also we note that $\Amp(\theta)$ only depends only the cohomology class $\{ \theta\}$.

If $\varphi$ and $\varphi'$ are two $\theta$-psh functions on $X$, then $\varphi'$ is said to be \emph{less singular} than $\varphi$, i.e., $\varphi \preceq \varphi'$, if they satisfy $\varphi\le\varphi'+C$ for some $C\in \mathbb{R}$.
We say that $\varphi$ has the same singularity as~$\varphi'$, i.e., $\varphi \simeq \varphi'$, if $\varphi \preceq \varphi'$ and $\varphi' \preceq \varphi$. The latter condition is easily seen to yield an equivalence relation, whose equivalence classes $[\varphi]$, $\varphi \in \psh(X, \theta)$, are called \emph{singularity types}.

A $\theta$-psh function $\varphi$ has \emph{minimal singularity type} if it is less singular than any other $\theta$-psh function. Such $\theta$-psh functions with minimal singularity type always exist, one can consider for example $V_\theta:=\sup \{ \varphi \ \theta\text{-psh},\, \varphi\le 0\text{ on } X \}$. Trivially, a $\theta$-psh function with minimal singularity type is locally bounded in $\Amp(\theta)$. It follows from \cite[Theorem~1.1]{DNT23} that $V_{\theta}$ is $C^{1, \bar{1}}$ in the ample locus ${\rm Amp}(\theta)$.

Given
\[\theta^1+{\rm dd}^c\varphi_1, \ \dots, \ \theta^p+{\rm dd}^c \varphi_p\] positive $(1,1)$-currents, where $\theta^j$ are closed smooth real $(1,1)$-forms, following the construction of Bedford--Taylor \cite{BT87} in the local setting, it has been shown in~\cite{BEGZ10} that the sequence of currents
\[
{\bf 1}_{\bigcap_j\{\varphi_j>V_{\theta_j}-k\}}\big(\theta^1+{\rm dd}^c \max(\varphi_1, V_{\theta_1}-k)\big)\wedge\dots\wedge (\theta^p+{\rm dd}^c\max(\varphi_p, V_{\theta_p}-k))
\]
is non-decreasing in $k$ and converges weakly to the so called \emph{non-pluripolar product}
\[
\big\langle \theta^1_{\varphi_1 } \wedge\dots\wedge\theta^p_{\varphi_p}\big\rangle .
\]
In the following, with a slight abuse of notation, we will denote the non-pluripolar product simply by $\theta^1_{\varphi_1 } \wedge\dots\wedge\theta^p_{\varphi_p}$.
When $p =n$, the resulting positive $(n,n)$-current is a Borel measure that does not charge pluripolar sets. Pluripolar sets are Borel measurable sets that are locally contained in the $-\infty$ locus of psh functions. As a consequence of \cite[Corollary 2.11]{BBGZ13} for any pluripolar set $A$, there exists $\psi\in \psh(X,\theta)$ such that $A\subset \{\psi=-\infty\}$.

For a $\theta$-psh function $\varphi$, the non-pluripolar product $\theta_\varphi^n$ is said to be the \emph{non-pluripolar Monge--Amp{\`e}re measure} of $\varphi$.

The volume of a big class $\{ \theta\}$ is defined by
\[
{\rm Vol}(\theta):= \int_{{\rm Amp}(\theta)} \theta_{V_\theta}^n.
\]
For notational convenience, we simply write ${\rm Vol}(\theta)$, but keeping in mind that the volume is a~cohomological constant. Note that ${\rm Vol}(\theta)>0$ as $\{\theta\}$ is big (cf.\ \cite{BEGZ10}).

A $\theta$-psh function $\varphi$ is said to have \emph{full Monge--Amp\`ere mass} if
\[
\int_X \theta_\varphi^n={\rm Vol}(\theta),
\]
and we then write $\varphi\in \mathcal{E}(X,\theta)$. By \cite[Theorem 1.16]{BEGZ10}, the set $\mathcal{E}(X,\theta)$ strictly contains the set of $\theta$-psh functions with minimal singularity type.

An important property of the non-pluripolar product is that it is local with respect to the plurifine topology (see \cite[Corollary~4.3]{BT87}, \cite[Section~1.2]{BEGZ10}).
This topology is the coarsest such that all qpsh functions on $X$ are continuous. For convenience, we record the following version of this result for later use.

\begin{Lemma} \label{lem: plurifine}
Fix closed smooth real big $(1,1)$-forms $\theta^1,\dots,\theta^n$. Assume that $\varphi_j$, $\psi_j$, $j=1,\dots,n$ are $\theta^j$-psh functions such that $\varphi_j =\psi_j$ on $U$, an open set in the plurifine topology. Then
\[
\id_{U} \theta^1_{\varphi_1} \wedge \dots \wedge \theta^n_{\varphi_n} = \id_{U} \theta^1_{\psi_1} \wedge\dots \wedge \theta^n_{\psi_n}.
\]
\end{Lemma}
Lemma \ref{lem: plurifine} will be referred to as the \emph{plurifine locality property}. We will often work with sets of the form $\{u<v\}$, where $u$, $v$ are quasi-psh functions. These are always open in the plurifine topology.

\subsection{Envelopes and model potentials}
Giving $f$ an upper semi-continuous function (usc for short) on $X$, we define the envelope of $f$ in the class ${\rm PSH}(X,\theta)$ by
\[
P_{\theta}(f) := \sup \{u\in {\rm PSH}(X,\theta) \colon u\leq f\},
\]
with the convention that $\sup\varnothing =-\infty$. Observe that $P_{\theta}(f)\in {\rm PSH}(X,\theta)$ if and only if there exists some $u\in {\rm PSH}(X,\theta)$ lying below $f$. Note also that $V_{\theta}=P_{\theta}(0)$, and that $P_{\theta}(f+C)= P_{\theta}(f)+C$ for any constant $C$.

 In the particular case $f=\min(\psi,\phi)$ for $\psi$, $\phi$ usc functions, we denote the envelope as
$ P_\theta(\psi,\phi):=P_\theta(\min(\psi,\phi))$.
We observe that $ P_\theta(\psi,\phi)= P_\theta(P_\theta(\psi),P_\theta(\phi))$, so without loss of generality, we can assume $\psi$, $\phi$ are two $\theta$-psh functions.

Starting from the ``rooftop envelope'' $ P_\theta(\psi,\phi)$, we introduce
\[P_\theta[\psi](\phi) := \big(\lim_{C \to \infty}P_\theta(\psi+C,\phi)\big)^*.\]
It is easy to see that $P_\theta[\psi](\phi)$ only depends on the singularity type of $\psi$. When $\phi = V_\theta$, we will simply write $P_\theta[\psi]:=P_\theta[\psi](V_\theta)$, and we refer to this potential as the \emph{envelope of the singularity type} $[\psi]$.

Since $\psi - \sup_X \psi \leq P_\theta[\psi]$, we have that $[\psi] \leq [P_\theta[\psi]]$ and typically equality does not happen. When $[\psi] = [P_\theta[\psi]]$, we say that $\psi$ has \emph{model singularity type}. In the (more particular) case~${\psi = P_\theta[\psi]}$, we say that $\psi$ is a \emph{model potential}.

It is worth to mention that given any $\theta$-psh function $\psi$ with positive mass, the associated envelope $P_\theta[\psi]$ is in fact a model potential \cite[Theorem~3.14]{DDL6}.
Also, we recall that by \cite[Remark~3.4]{DDL6}, we know that $\int_X \theta_\psi^n = \int_X \theta_{P_\theta[\psi]}^n$.

From now on (unless otherwise stated), $\phi$ will denote a model potential with positive mass, i.e., $\int_X \theta_\phi^n>0$. We say that a $\theta$-psh function $\varphi$ has $\phi$-relative minimal singularities if $\varphi \simeq \phi$.

\begin{Definition}
Given a model potential $\phi$, the relative full mass class $\mathcal{E}(X,\theta,\phi)$ is the set of all $\theta$-psh functions $u$ such that $u$ is more singular than $\phi$ and $\int_X \theta_u^n=\int_X \theta_{\phi}^n$.
\end{Definition}

\begin{Proposition}\label{mass}
Assume $\{ \theta \}$ is a big cohomology class, and $0 \leq m \leq V$ with $V = \int_X \theta_{V_\theta}^n$. Then there exists a model potential $ u \in \PSH(X,\theta)$ such that $\int_X \theta_u^n = m$.

Moreover, $V_\theta$ is the only model potential with Monge--Amp\`ere mass~$V$, whereas there are infinitely many model potentials with Monge--Amp\`ere mass $m < V$.
\end{Proposition}

\begin{proof} For the first statement, we need to treat the $1$-dimensional case separately. When $n=1$, the big cohomology class $\{\theta\}$ is actually K\"ahler, hence we can work with a K\"ahler form $\omega\in \{\theta\}$ with $\int_X \omega = V$ as a reference form.
Fix $a\in X$, then the measures $\omega$ and $V \delta_a$ define cohomologus closed positive $(1,1)$-currents. Hence, there exists $\psi_a \in \SH(X,\omega)$ such that $\omega + {\rm dd}^c \psi_a = V \delta_a$ and~$\sup_X \psi_a = 0$.
Note that on a fixed local holomorphic coordinate chart centered in $a$ we have that~$\psi_a $ writes as the sum of $ V \log(|z|)$ and a smooth function. Thus, $ \psi_a(a) = - \infty$. Observe moreover that the non-pluripolar product $( \omega + {\rm dd}^c \psi_a)$ is given by the product of the current~${\omega + {\rm dd}^c \psi_a}$ with the characteristic function $ {\bf 1}_{ \{ \psi_a > - \infty \} }$. This implies that $\omega + {\rm dd}^c \psi_a$ has mass zero. Since for $t \in [0,1]$, we have $ \omega +t {\rm dd}^c \psi_a \geq (1-t) \omega \geq 0$, the function
\[m(t):= \int_X (\omega+t{\rm dd}^c\psi_a)\in [0, V]\]
 defined on $[0,1]$ is continuous and $m(0)=V$, $m(1)=0$, we arrive at the conclusion.

Assume now $n\geq 2$. Let us consider $\pi \colon Y \to X$ the blow up of $X$ at a point $p \in X$ with exceptional divisor $E$. Let $[E]$ be the current of integration along $E$.

Note that the cohomology class $\{ \pi^*\theta \} $ is big and so is $\{\pi^*(\theta) - t [E]\}$ for $ |t| < \varepsilon$ with~$\varepsilon$ small enough.
Observe that the pseudoeffective cone of a compact K\"ahler manifold does not contain lines. In fact, let $\alpha $ and $\gamma$ be smooth closed real $(1,1)$-forms and assume that
${\{ \alpha \} + t \{ \gamma \}}$ is a pseudoeffective class for all
$t \in \mathbb{R}$.
Then for $k \in { \N}$, \smash{$ \frac{ \{ \alpha \} }{k}+ \{ \gamma \} = \frac{1}{k}( \{ \alpha \}+ k\{ \gamma \})$} and \smash{$\frac{ \{\alpha \}}{k}+ \{-\gamma \} = \frac{1}{k}( \{ \alpha \} + k\{-\gamma \})$} are pseudoeffective, but the pseudoeffective cone is closed, then letting $k\rightarrow +\infty$, we get that both
$\{ \gamma \}$ and $\{ - \gamma \} $ are pseudoeffective classes, meaning that~${\{ \gamma \} = 0}$.

Therefore, for large $t$, $\{\pi^*(\theta) - t [E] \}$ is no longer big. Let
\[a = \sup \{ t \in \mathbb{R}, \ t > 0 \colon \mbox{ the class $\{ \pi^*(\theta) - t[E] \}$ is big} \}.
\]
Then $ 0< a < +\infty $ and
$\{\pi^*\theta - a [E]\}$ is pseudoeffective but not big.
In other words, if $\gamma$ is a~smooth form representing
$\{ \pi^*\theta - a [E] \}$, then
\[\int_Y (\gamma+{\rm dd}^c V_\gamma)^n = 0. \]

Now, $\gamma+{\rm dd}^c V_\gamma + a[E]$ is a positive $(1,1)$-current and, since the non-pluripolar product does not put mass of analytic subsets, we also get
\[\int_Y ( \gamma +{\rm dd}^c V_\gamma + a[E])^n = 0. \]

On the other hand, $\{\gamma+{\rm dd}^c V_\gamma + a[E]\}= \{ \pi^*\theta \}$, hence by \cite[Proposition 1.2.4]{BouThesis}, there exists~${u_1\in \PSH(X,\theta)}$ such that
$\pi^*\theta_{u_1} = \gamma +{\rm dd}^c V_\gamma + a[E]$.
In particular, $ 0=\int_Y (\pi^*\theta_{u_1})^n =\int_X \theta_{u_1}^n $.
Therefore, since the function
\[m(t) := \int_X ( \theta + (t u_1 + (1-t) V_\theta))^n \in [0, V] \]
defined in $[0,1]$ is continuous, $m(0)=V$, $m(1)=0$, then for $0 \leq m \leq V$ there exists $t_0\in [0,1]$ such that $\int_X (\theta + {\rm dd}^c ( t_0 u_1 + (1-t_0) V_\theta)^n = m$. On the other hand, by \cite[Remark 3.4]{DDL6}, $u_t := tu_1 + (1-t)V_\theta$ and $P[u_t]$ have the same mass for every $t \in [0,1]$.

For the last statement, we observe that, since $P[u_t] \leq V_\theta$ for all $t \in [0,1]$, by \cite[Theorem~3.14]{DDL6}, $V_\theta$~is the only model potential with Monge--Amp\'ere mass $V$. Now, by~\cite[Corollary~1.18]{BouThesis}, the Lelong number of the function~$u_1$ at~$p$ is strictly positive. Therefore, for $0 < t \leq 1$ also the Lelong number of $u_t$ is strictly positive at~$p$. By \cite[Lemma~5.1]{DDL6}, so is the Lelong number of~$P[u_t]$. Therefore, fixing $0 \leq m < V$ and varying the point $p$, we obtain infinitely many distinct model potentials all of the same Monge--Amp\`ere mass $m$.
\end{proof}

As pointed out in \cite{GZ07}, and then in the relative setting in \cite{Gup23}, it is natural to consider weighted subspaces of $\mathcal E(X,\theta,\phi)$.

A \emph{weight} is a continuous strictly increasing function $\chi\colon [0,+\infty) \rightarrow [0,+\infty)$ such that $\chi(0)=0$ and $\chi(+\infty)=+\infty$. Denote by $\chi^{-1}$ its inverse function, i.e., such that $\chi\big(\chi^{-1}(t)\big)= t$ for all~${t\geq 0}$.

We fix $\phi$ a model potential and we let $\mathcal{E}_{\chi}(X,\theta,\phi)$ denote the set of all $u\in \mathcal{E}(X,\theta,\phi)$ such that
\[
E_{\chi}(u,\phi):=\int_X \chi(|u-\phi|) \theta_u^n <\infty.
\]
When $\phi=V_{\theta}$, we denote $\mathcal{E}(X,\theta)=\mathcal{E}(X,\theta,V_{\theta})$, $\mathcal{E}_{\chi}(X,\theta)=\mathcal{E}_{\chi}(X,\theta,V_{\theta}$) and $E_{\chi}(u)=E_{\chi}(u,V_{\theta})$.

Compared to \cite{GZ07}, we have changed the sign of the weight, but the weighted classes are the same.

Also, in the special case $\chi(t)=t^p$, $p>0$, we simply denote the relative energy class with~${\mathcal{E}^p(X,\theta,\phi)}$
and the corresponding relative energy $E_p(u,\phi)$.

These weighted Monge--Amp\`ere energy classes covers the all class $\mathcal{E}$, i.e.,

\begin{Lemma}\label{energy union} We have that
$\mathcal{E}(X,\theta,\phi)=\bigcup_{\chi}\mathcal{E}_\chi(X,\theta,\phi)$.
\end{Lemma}

\begin{proof}We give here the proof for reader's convenience.
The inclusion $ \supseteq $ clearly follows from the definition. In order to prove the other inclusion we need to show that, given $\varphi \in \mathcal{E}(X,\theta,\phi)$ normalized with $\sup_X \varphi =-1$, there exists a weight function $\chi$ such that $\int_X \chi (\phi-\varphi) {\rm d}\mu < +\infty$, where $\mu:=\theta_\varphi^n$. Also, for convenience we normalize~$\mu$ to have $\mu(X)=1$.

Set \smash{$\psi(t):= \frac{1}{\mu(\{\varphi <\phi -t\})}$}, and observe that $\psi\colon\R^+\rightarrow \R^+$ is an increasing function such that $\psi(0)=1$ and $\lim_{t\rightarrow +\infty} \psi(t) =+\infty$. The latter property follows from the definition of the non-pluripolar product. We also note that, since $t\rightarrow \mu(\{\varphi <\phi -t\})$ is a decreasing function, its discontinuity locus is at most countable. In particular, $\psi$ is continuous for all $t\in \R^+\setminus J$.

We now consider $h\colon \R^+\rightarrow \R^+$ such that $h> 0$ satisfying:
\[{\rm (a)} \ \int_1^{+\infty} h(s) {\rm d}s <+\infty, \qquad {\rm (b)} \ \int_1^{+\infty} {h(s)} \psi(s) {\rm d}s =+\infty.\]
 Such a function $h$ exists as we show below. We choose a sequence of distinct points $\{t_k\}_{k\in \N}\in \R^+$ such that $\psi(t_k) \geq k$. Such a sequence exists since $\psi$ is increasing and $\psi\rightarrow +\infty $ as $t$ goes to~$+\infty$. Also, we can take $t_1=1$ since $\psi(1)\geq \psi(0)=1$. We then set $I_k:=[t_{k}, t_{k+1}[$ and we define $h$ to be the piece-wise continuous function constant on each $I_k\colon h|_{I_k}= k^{-2} (t_{k+1}-t_k)^{-1} $. Then
 \[\int_1^{+\infty} h(s) {\rm d}s= \sum_{k\geq 1} \frac{1}{k^2} <+\infty \qquad \text{and}\qquad \int_1^{+\infty}h(s)\psi(s) {\rm d}s\geq \sum_{k\geq 1}\frac{1}{k} =+\infty.\]

We then define the function $\chi\colon \R^+\rightarrow \R^+$ as \smash{$\chi(t):= \int_0^ t\frac{h(s)}{\mu(\{\varphi <\phi -s\})} {\rm d}s$}. We infer that the fundamental theorem of calculus holds almost everywhere. Indeed, for $t\notin J\cup \{t_k\}_k$, given $\varepsilon>0$ (small enough), the mean value theorem ensures that there exists $c_\varepsilon \in (t, t+\varepsilon)$ such that
\[\frac{\chi (t+\varepsilon)-\chi (t)}{\varepsilon} = \frac{1}{\varepsilon} \int_t^{t+\varepsilon} h(s) \psi(s) {\rm d}s = h(c_\varepsilon) \psi(c_\varepsilon).\]
Sending $\varepsilon\to 0$, and using the continuity of $h\cdot \psi$ on $\R^+\setminus (J\cup \{t_k\}_k)$, we get that $(\chi'(t))^+=h(t)\psi(t)$.

The same arguments work for left derivative, giving $(\chi'(t))^+=(\chi'(t))^-=\chi'(t)=h(t)\psi(t)$ almost everywhere, as we wanted.
By definition, $\chi$ is strictly increasing (since $h> 0$) and $\chi(+\infty)=+\infty$, thanks to condition (b). Moreover, setting $\varepsilon:=\chi (1)>0$, we have
\begin{align*}
\int_X \chi (\phi-\varphi) {\rm d}\mu = & \int_{\varepsilon}^\infty \mu(\{ \chi (\phi-\varphi) >t\}) {\rm d}t=\int_{\varepsilon}^\infty \mu(\{ \phi-\varphi >\chi^{-1}(t)\}) {\rm d}t\\
=& \int_{1}^\infty \chi'(s) \mu(\{ \phi-\varphi >s\}) {\rm d}s= \int_{1}^\infty h(s) {\rm d}s <+\infty,
\end{align*}
thanks to condition (a). This concludes the proof.
\end{proof}

We conclude here with a simple lemma needed in the following that we could not find in this precise form in the literature:

\begin{Lemma}\label{energy increases}
 Let $u\in \mathcal{E}_{\chi}(X,\theta,\phi)$ and $u_j:=\max (u,\phi -j)$. Then $u_j \in \mathcal{E}_{\chi}(X,\theta,\phi)$ and $E_{\chi}(u_j,\phi) \nearrow E_{\chi}(u,\phi)$.
\end{Lemma}
\begin{proof}
 For $j\in \N$, by Lemma \ref{lem: plurifine}, we have $\id_{\{u>\phi -j\}} \theta_{u_j}^n= \id_{\{u>\phi-j\}} \theta_{u}^n$. Since $\int_X \theta_{u}^n=\int_X \theta_{u_j}^n=\int_X \theta_\phi^n$, we obtain
\begin{align*}
0\leq \limsup_{j\to +\infty} \chi(j)\int_{\{u\leq \phi-j\}}\!\theta_{u_j}^n = \limsup_{j\to +\infty} \chi(j)\int_{\{u\leq \phi-j\}}\!\theta_{u}^n
\leq \limsup_{j\to +\infty} \int_{\{u\leq\phi-j\}}\! \chi (\phi-u)\theta_{u}^n =0.
\end{align*}
Now,
\begin{align*}
\int_X \chi (|\phi-u_j|) \theta_{u_j}^n &= \int_{\{ u > \phi-j\}} \chi (|\phi-u_j|) \theta_{u_j}^n+ \int_{\{ u \leq \phi-j\}} \chi (\phi-u_j) \theta_{u_j}^n \\
 &= \int_{\{ u > \phi-j\}} \chi (|\phi-u|) \theta_{u}^n+ \int_{\{ u \leq \phi-j\}} \chi (\phi-u_j) \theta_{u_j}^n.
\end{align*}
The conclusion follows letting $j\rightarrow +\infty$.
\end{proof}

\section{Entropy}
 We recall that given two positive probability measures $\mu$, $\nu$, the relative entropy $\Ent(\mu, \nu)$ is defined as
\[
\Ent(\mu, \nu) := \int_X \log\left (\frac{{\rm d}\mu}{{\rm d}\nu}\right) {\rm d}\mu,
\]
if $\mu$ is absolutely continuous with respect to $\nu$, and $+\infty$ otherwise.

\begin{Remark}
Let $\mu$, $\nu$ positive probability measure with $\mu := f \nu $ absolutely continuous with respect to $\nu$. Then $\Ent(\mu,\nu) < + \infty$ if and only if $f\log f \in L^1(X,\nu)$, in fact if $f < 1$, $f \log f$ is bounded, and if $f \geq 1$, $f \log f \geq 0$. \label{L^1}
\end{Remark}

Once and for all, we normalize the K\"ahler form $\omega$ such that $\int_X \omega^n=1$. We consider $u\in \psh(X, \theta)$ such that $\theta_u^n=f \omega^n$, $0 \leq f$ and $m_u:=\int_X \theta_u^n >0$. Then $ u\in \mathcal{E}(X, \theta, \phi)$ for the model potential $\phi=P_\theta[u]$ \cite[Theorem 1.3]{DDL2}, and $m_u^{-1}\theta_u^n$ is a probability measure. We then define the $\theta$\emph{-entropy} of $u$ as
\begin{align}
 {\rm Ent}_\theta(u)&:= {\rm Ent}\big( m_u^{-1}\theta_u^n, {\omega^n}\big)=\int_X \log \left( \frac{m_u^{-1}\theta_u^n}{\omega^n} \right) m_u^{-1}\theta_u^n\nonumber\\
 &= m_u^{-1} \int_X f\log f \omega^n-\log m_u.\label{eq: entropy}
 \end{align}
By Jensen’s inequality, we have ${\rm Ent}_\theta(u)\geq 0$. Also, observe that the definition of the $\theta$-entropy does depend on the chosen volume form $\omega^n$ but its finiteness does not.

Also, the expression in~\eqref{eq: entropy} coincides with the definition of entropy in \cite{DGL20} when $P[u]=V_\theta$, i.e., when $u\in \mathcal{E}(X, \theta)$. The definition in~\eqref{eq: entropy} is indeed a generalisation that allows to consider any $\theta$-psh function not necessarily of full mass.

More generally, given two $\theta$-psh functions $u$, $v$ with $m_u, m_v>0$ we define
\[{\rm Ent}_\theta(u,v) :={\rm Ent}\big( m_u^{-1}\theta_u^n, m_v^{-1}\theta_v^n\big). \]
Also, if no confusion can arise, we simply write ${\rm Ent} (u) $ and ${\rm Ent} (u,v) $.

\begin{Definition}\label{def: entropy}
 We say that $u\in \psh(X, \theta)$ with $m_u>0$ has \emph{finite $\theta$-entropy} if ${\rm Ent}_\theta(u)<+\infty$. We denote by ${\rm Ent}(X, \theta)$ the set of all $\theta$-psh functions having finite $\theta$-entropy.
\end{Definition}
 By \eqref{eq: entropy}, ${\rm Ent}_\theta(u)<+\infty$ if and only if $\theta_u^n=f\omega^n$ and $\int_X f\log f\omega^n<+\infty$ or equivalently $\int_X (f+1)\log (f+1)\omega^n<+\infty$.

We start with the following observations ensuring that the set ${\rm Ent}(X, \theta)$ is not empty.

\begin{Proposition}\label{env} Let $\phi$ be a model potential with $\int_X \theta_\phi^n >0$. The following hold:
\begin{itemize}\itemsep=0pt
 \item[$(i)$] $\Ent_\theta(\phi)<+\infty$.

\item[$(ii)$] Let $\varphi\in \mathcal{E}(X, \theta, \phi)$ such that $\theta_\varphi^n =f \omega^n$ with $f\in L^p(X)$, $p>1$. Then ${\rm Ent}_\theta(\varphi)<+\infty$.

\item[$(iii)$] Let $\varphi\in \psh(X, \theta)$ have analytic singularities with smooth remainder such that $m_\varphi>0$. Then ${\rm Ent}_\theta(\varphi)<+\infty$.

\item[$(iv)$] Let $f_1, \dots,f_k$ be $C^{1,\bar{1}}$ functions on $X$, i.e., bounded functions with bounded $($distributional$)$ Laplacian. Then
 ${\rm Ent}_\theta(P_{\theta}(f,\dots,f_k))<+\infty $.
\end{itemize}
\end{Proposition}

\begin{proof}
By \cite[Theorem~3.6]{DDL6},
\[\theta_\phi^n\leq {\bf 1}_{\{\phi=0\}} \theta^n\leq {\bf 1}_{\{\phi=0\}} C \omega^n\] for some positive constant $C$.
In particular, $\theta_\phi^n=g\omega^n$ for some $g\in L^\infty (X)$, $0\leq g \leq C$. This proves~(i). For (ii), we observe that since $p-1>0$, we have that
\[\int_X f \log f \omega^n \leq C \int_X f |f|^{p-1} \omega^n<+\infty. \]
 This implies that
 \[0\leq \Ent_\theta(\varphi)= -\log m_\phi + m_\phi^{-1} \int_X f \log f \omega^n<+\infty.\]

Given $\varphi\in \psh(X, \theta)$ with analytic singularities with smooth remainder, it follows from \cite[Proposition~4.36]{DDL2} that $\theta_\varphi^n= f\omega^n$ with $f\in L^p(X)$. The previous step then gives (iii).

We now prove the last statement. We first assume $k =1 $ and $f_1$ quasi-psh. If we let $\bar{\theta} = \theta + {\rm dd}^c f_1$, then
\[
P_\theta(f_1) = P_{\bar{\theta}}(0)+ f_1 \qquad \text{and}\qquad  \theta_{P_\theta(f_1)}^n = \bar{\theta}_{P_{\bar{\theta}}(0)}^n.
\]
 Observe that since $f_1$ is $C^{1,\bar{1}}$ and~quasi-psh, ${\rm dd}^c f_1$ is bounded. This means that $\bar{\theta}$ is a $(1,1)$-form with bounded coefficients.

 By \cite[Corollary 3.4\,(i)]{DNT21}, we have
 \smash{$\bar{\theta}_{P_{\bar{\theta}}(0)}^n= g_1 \bar{\theta}^n$} for some non-negative bounded function $g_1$ on~$X$. Since $ \bar{\theta}$ has bounded coefficients, we can ensure that there exists a non-negative bounded function $g_2$ such that $\bar{\theta}_{P_{\bar{\theta}}(0)}^n = g_2 \omega^n$. It then follows that $P_{\theta}(f_1)$ has finite entropy with respect to~$\omega$.

 We now treat the general case of $k$ functions $f_1, \dots,f_k$ which are assumed to be only $C^{1,\bar{1}}$. We choose $C > 0$ such that $\theta \leq C \omega$ and we claim that
\[P_\theta(f_1,\dots,f_k) = P_{\theta}(P_{C \omega}(f_1,\dots,f_k)).\]
Indeed, $P_\theta(f_1,\dots,f_k) \geq P_{\theta}(P_{C \omega}(f_1,\dots,f_k))$ since $ P_{C \omega}(f_1,\dots,f_k)\leq \min (f_1,\dots,f_k)$; for the other inequality, we have that
\[P_{\theta}(f_1,\dots,f_k) \leq \min (f_1,\dots,f_k),\]
and that $P_{\theta}(f_1,\dots,f_k)$ is $\theta$-psh hence it is also $C \omega$ plurisubharomnic. This implies
\[P_{\theta}(f_1,\dots,f_k) \leq P_{C \omega}(f_1,\dots,f_k).\]
If we now apply $P_\theta$ to both sides of the above inequality, we find
\[P_{\theta}(f_1,\dots,f_k) \leq P_\theta(P_{C \omega}(f_1,\dots,f_k)).\]
 By \cite[Theorem 2.5]{DR16}, $P_{C \omega}(f_1,\dots,f_k)$ is $C^{1,\bar{1}}$ on $X$ and quasi-psh. We can then apply the previous step to conclude.
\end{proof}

Since, by \cite[Theorem 5.19]{DDL6}, we know that $[\phi]=[\varphi]$, the above results could make a reader ask whether the property of having finite entropy is stable in the singularity class, i.e., if given~${\varphi_1,\varphi_2\in \PSH(X,\theta)}$ with $[\varphi_1]=[\varphi_2]$, then $\Ent_\theta(\varphi_1)<+\infty$ iff $\Ent_\theta(\varphi_2)<+\infty$. The answer is negative as the following example shows:

 \begin{Example} \label{exa: No Ent}
Let $U\subset X$ be a local chart and write $\omega={\rm dd}^c \rho$ in $U$ and define
\[u= \frac{1}{C } \chi \cdot \max(\log \|z\|, 0)-\rho,\]
where $\chi$ is a cut-off function such that $\chi\equiv 1$ on $\mathbb{B}(r_1)$ and $\chi\equiv 0 $ on $U\setminus \mathbb{B}(r_2)$, for $r_1, r_2>0$ small enough so that $\mathbb{B}(r_2)\Subset \mathbb{B}(r_2)\Subset U$. Without loss of generality, we may assume $r_1=1$,~$r_2=2$.
Choosing $C$ big enough, $u$ induces a $\omega$-psh function which we note by $\tilde{u}$.
Then $\tilde{u}$ is bounded, hence $[\tilde{u}]=[0]$. But
\begin{gather*}
(\omega+{\rm dd}^c \tilde{u})^n \\
\quad= C^{-n} ({\rm dd}^c (\max(\log \|z\|, 0))^n + \sum_{j=0}^{n-1} \binom{n}{j} C^{-j} \omega^{n-j} \wedge ({\rm dd}^c (\max(\log \|z\|, 0))^j \quad \text{in $\mathbb{B}(1)$}
\end{gather*}
and the measure $({\rm dd}^c (\max(\log \|z\|, 0))^n$ is the normalized Lebesgue measure on the torus $\big(S^1\big)^n\!\subset \C^n$ (that is, a real analytic subspace of real dimension~$n$). It then follows that $(\omega+{\rm dd}^c \tilde{u})^n$ is not even absolutely continuous with respect to $\omega^n$.
 \end{Example}

\subsection{Entropy and energy}
It was proved in \cite{DGL20, DNL21} that \[\Ent(X,\theta)\cap \mathcal{E}(X,\theta) \subset \cE^{\frac{n}{n-1}}(X,\theta).\] As will be shown below, one can extend these results to the case of prescribed singularities.

We start with an integrability result of Moser--Trudinger type for general weights:

\begin{Theorem}
 	\label{thm: rel Moser--Trudinger inequality weight}
 	Let $\phi$ be a model potential with $m_\phi>0$. Let $\chi_1\colon \R^+\rightarrow \R^+ $ be a weight. Let \smash{$\chi_2(t) := \int_0^t \chi_1(s)^{\frac{1}{n}} {\rm d}s$}. Then there exist $c>0$, $C>0$ depending on $X$, $\theta$, $n$, $m_\phi$ and $\omega$ such that, for all $\varphi\in \mathcal{E}_{\chi_1}(X,\theta, \phi)$ with $\sup_X \varphi=-1$, we have
 	\begin{equation}
 	\int_X \exp \big( c E_{\chi_1}(\varphi, \phi)^{-\frac{1}{n}} \chi_2(\phi-\varphi)\big)\omega^n \leq C.
 	\label{eq: Moser} \end{equation}
 \end{Theorem}

 \begin{proof}
 We first note that if we replace $\chi_1$ by $\alpha \chi_1$ with $\alpha$ positive constant the left-hand side of~\eqref{eq: Moser} does not change, so we may assume $\chi_1(1) = 1$. 	
We claim that it suffices to prove the above inequality for $\varphi$ with relative minimal singularities, i.e., $[\varphi]=[\phi]$. Indeed, given~${\varphi\in \psh(X, \theta)}$, $\varphi_j:=\max(\varphi, \phi-j)$ has the same singularity type as $\phi$. This would mean that
 \[\int_X \exp \big( c E_{\chi_1}(\varphi_j, \phi)^{-\frac{1}{n}} \chi_2(\phi-\varphi_j)\big)\omega^n \leq C.\]
Moreover, it follows from Lemma \ref{energy increases} that $ E_{\chi_1}(\varphi_j, \phi) \nearrow E_{\chi_1}(\varphi, \phi) $ as $j\rightarrow +\infty$. Fatou's lemma will then give the desired estimate.

Thus assume that $\varphi$ has relative minimal singularities.
Let $\psi=-a\chi_2(\phi-\varphi)+\phi$ where $a>0$ is a small constant to be suitably chosen. Then define $u=P_\theta(\psi)$ and $v=-\gamma_2 (\phi-u) +\phi$, where~${\gamma_2 \colon \mathbb{R}^+ \rightarrow \mathbb{R}^+}$ denotes the inverse function of $a\chi_2$, which is concave and increasing. We observe that $u$ is a $\theta$-psh function with the same singularity type as $\phi$ and that
\[v=-\gamma_2 (\phi-u) +\phi \leq -\gamma_2(\phi-\psi) +\phi= -\gamma_2(a\chi_2(\phi-\varphi)) +\phi = \varphi\] with equality on the contact set $\mathcal{C}=\{u=\psi\}$.

A simple computation gives
 \begin{flalign*}
 \theta +{\rm dd}^c v & = \gamma_2'(\phi-u) \theta_u + (1-\gamma_2'(\phi-u)) \theta_{\phi} - \gamma_2''(\phi-u) d(\phi-u) \wedge d^c (\phi-u)\\
&\geq \gamma_2'(\phi-u) \theta_u+ (1-\gamma_2'(\phi-u)) \theta_{\phi},
 \end{flalign*}
where in the above we used that $\gamma_2$ is concave. We consider the set $G:=\{\gamma_2'(\phi-u)< 1\}$. We are going to show by contradiction that for a suitable choice of $a$ we have $G \neq X$. So, we assume $G= X$.
It then follow that $v$ is $\theta$-psh and, by construction, we infer that $v$ has the same singularity type as $\phi$.

Recall that $\sup_X \varphi=\sup_X(\varphi-\phi)$ as $\phi$ is a model potential \cite[Lemma 3.5]{DDL6}, hence $\phi-\varphi \geq 1$. This implies that $\chi_1$ is increasing and $\chi_1(1) = 1$ giving that
 \[E_{\chi_1}(\varphi, \phi) = \int_X \chi_1(\phi-\varphi) \theta_{\varphi}^n \geq \chi_1(1)\int_X \theta_{\varphi}^n= \int_X \theta_{\phi}^n=m_\phi.\]

By \cite[Theorem 2.7]{DDL6}, the non-pluripolar Monge--Amp\`ere measure $(\theta+{\rm dd}^c u)^n$ is supported on $\mathcal{C}$, hence
\begin{align*}
	(\theta+{\rm dd}^c v) &\geq \gamma_2'(\phi-u) (\theta+{\rm dd}^c u)
		 = (a\chi_2'(\phi-\varphi))^{-1} (\theta+{\rm dd}^c u)\quad \text{on $\mathcal{C}$}.
\end{align*}
The last identity follows from the fact that, since $(a\chi_2)'(\gamma_2(t))\gamma_2'(t)=1$ and $\gamma_2(\phi-u)=\phi-\varphi$ in $\mathcal{C}$, we have $\gamma_2'(\phi-u)= (a\chi_2'(\phi-\varphi))^{-1}$ in $\mathcal{C}$.

It follows from \cite[Lemma 5.19]{DDL6} that the above inequality between positive currents implies an inequality between the non-pluripolar measures (observe that this is not trivial since $(a\chi_2'(\phi-\varphi))^{-1} (\theta+{\rm dd}^c u)$ is not closed).
Thus, we can infer that
\begin{equation*}
 {\bf 1}_{\mathcal{C}} \;(a\chi_2'(\phi-\varphi))^{-n} (\theta+{\rm dd}^c u)^n
	\leq {\bf 1}_{\mathcal{C}}(\theta+{\rm dd}^c v)^n \leq {\bf 1}_{\mathcal{C}}(\theta+{\rm dd}^c \varphi)^n,
\end{equation*}
where the last inequality follows from \cite[Lemma 4.5]{DDL5}. The above is equivalent to
\begin{align}
 (\theta+{\rm dd}^c u)^n &={\bf 1}_{ \mathcal{C}} (\theta+{\rm dd}^c u)^n \leq {\bf 1}_{ \mathcal{C}} (a\chi_2'(\phi-\varphi))^{n}(\theta+{\rm dd}^c \f)^n\nonumber\\
 &\leq a^n \chi_1(\phi-\varphi)(\theta+{\rm dd}^c \f)^n, \label{eq: stimabase}
\end{align}
where in the last inequality we used that $(\chi_2')^n=\chi_1$.

We now choose $a$ so that
\[
\beta a^n \int_X \chi_1(\phi-\varphi) (\theta+{\rm dd}^c \varphi)^n=\beta a^n E_{\chi_1}(\varphi, \phi) = m_\phi
\]
 with $\beta > 1$.
Observe that $a\in(0,1)$ since we observed that $E_{\chi_1}(\varphi, \phi)\geq m_\phi$.
Integrating both sides of \eqref{eq: stimabase} over $X$, we obtain
\[
\int_X (\theta+{\rm dd}^c u)^n \leq \frac{m_\phi}{\beta}.
\]
Since $\int_X \theta_u^n=m_\phi$, we arrive at a contradiction.

So we can infer that $G\neq X$, or equivalently that $G^c\neq \varnothing$. This means that there exists~${x_0\in X}$ such that $(\phi-u)(x_0)\leq \tau_2(1)$ where $\tau_2$ is the inverse function of $\gamma_2'$ which we note to be decreasing (since so in $\gamma_2'$). In particular,
\[\sup_X u =\sup_X (u-\phi) \geq -\tau_2(1).\]

Applying the uniform version of Skoda's integrability theorem \cite{Zer01} to $\PSH(X,C\omega)$ for ${C>0}$ such that $\theta\leq C\omega$, we know that there exist uniform constants $c_0, C_0>0$ such that, for all~${h \in \PSH(X,\theta)}$ with $\sup_X h=0$ we have $\int_X {\rm e}^{-c_0 h} \omega^n \leq C_0$. For $h:=u+\tau_2(1)$, we have $\sup_X h \geq 0$, hence
\begin{align*}
 \int_X {\rm e}^{c_0 (a \chi_2(\phi-\varphi) -\tau_2(1))} \omega^n &= \int_X {\rm e}^{c_0 (-\psi+\phi -\tau_2(1))} \omega^n \leq \int_X {\rm e}^{-c_0 (u +\tau_2(1))} \omega^n
\\&= \int_X {\rm e}^{-c_0 (h-\sup_X h) - c_0\sup_X h} \leq C_0,
\end{align*}
where in the first inequality we used $\phi\leq 0$ and $u\leq \psi$.

Let us now give a more explicit expression of $\tau_2(1)$.
By definition, $\tau_2$ is such that
\[s=\tau_2 \left( \frac{1}{a\chi_2'(\gamma_2(s))} \right).\] Now we want to understand $\tau_2(t)$, hence we want to find $s$ such that \smash{$\tfrac{1}{a\chi_2'(\gamma_2(s))}=t $}. This means~${ (at)^{-1}=\chi_2'(\gamma_2(s)) = (\chi_1(\gamma_2(s)))^{1/n} }$ (since $(\chi_2')^n=\chi_1$). Therefore, $\gamma_1((at)^{-n}) =\gamma_2(s)$, where $\gamma_1$ is the inverse function of $\chi_1$. It the follows that $ \tau_2(t)= a\chi_2(\gamma_1((at)^{-n}))$.
In particular, $\tau_2(1)= a\chi_2(\gamma_1(a^{-n}))$.

Also, letting $s:=\gamma_1 (a^{-n})$, we have $a^{-1} = (\chi_1(s))^{\frac{1}{n}}$ so that $a=a(s) = \frac{1}{\chi_2'(s)}$ and
\begin{gather} \tau_2(1) = \frac{\chi_2(s)}{\chi_2'(s)}. \label{Bo} \end{gather}
Note that since $a\in(0,1)$, then $s > 1$ ($\gamma_1$ is increasing and $\gamma(1)=1$).

Next, setting \[K:= \{ x \in X \colon a \chi_2(\phi - \varphi) \leq 2(\phi - \varphi) \},\]
we claim that
\[2^{-1} a \chi_2(\phi - \varphi) \leq a \chi_2(\phi - \varphi) - \tau_2(1) \quad \text{on $X \setminus K$},\]
that is, $a \chi_2(\phi - \varphi) \geq 2 \tau_2(1)$ on $X \setminus K$.

We observe that, since $\chi_2$ is convex and $\chi_2(0) = 0$, the set $E:=\{ t \in \mathbb{R}^+ \colon a\chi_2(t) \leq 2 t \}$ is a~closed convex set containing $0$, i.e., it is an interval of the form $[0, \lambda]$ with $0 \leq \lambda \leq + \infty$. Hence, $x\in X\setminus K$ if and only if $(\phi-\varphi)>\lambda$.
We then need to prove that if $ a\chi_2( \phi - \varphi) > a\chi_2( \lambda)$,
then
$ a \chi_2( \phi - \varphi) \geq {2 \tau_2(1)}$. It is then sufficient to prove that $a\chi_2( \lambda) \geq 2 \tau_2(1)$. The above is equivalent to $ \lambda \geq \gamma_2(2 \tau_2(1)) \geq 0$. By definition of $\lambda$, this means that $\gamma_2(2 \tau_2(1))\in E$, i.e.,
\[\tau_2(1) \leq \gamma_2(2 \tau_2(1)),\qquad \text{or} \qquad a \chi_2(\tau_2(1)) \leq 2 \tau_2(1).\]
By (\ref{Bo}) and $a=1/\chi_2'(s)$, the above inequality gives
\[\chi_2\left(\frac{\chi_2( s)}{\chi_2' (s)} \right) \leq 2 \chi_2(s).\]
Since $\chi_2$ is convex and $\chi_2(0) = 0$, we have
$\chi_2(s) \leq s \chi_2'(s)$.
Then we obtain
\[ \chi_2\left(\frac{\chi_2( s)}{\chi_2' (s)} \right) \leq \chi_2(s) \leq 2 \chi_2(s),\]
and the claim is proved.

It follows that
\begin{align*}
\int_X {\rm e}^{\frac{c_0}{2} a \chi_2(\phi-\varphi)} \omega^n &\leq \int_K {\rm e}^{c_0 (\phi-\varphi)} \omega^n + \int_{X\setminus K}{\rm e}^{c_0(a\chi_2(\phi-\varphi) -\tau_2(1))} \omega^n\\
 &\leq \int_K {\rm e}^{-c_0 \varphi} \omega^n + \int_{X\setminus K}{\rm e}^{c_0(a\chi_2(\phi-\varphi) -\tau_2(1))} \omega^n \\
 &\leq\int_K {\rm e}^{-c_0 (\varphi+1)+c_0} \omega^n + C_0\leq C_0({\rm e}^{c_0}+1),
\end{align*}
wherein the last inequality we applied the uniform Skoda theorem to the function $\varphi + 1$. Recalling that $\beta a^n E_{\chi_1}(\varphi, \phi) = m_\phi$, the result then follows with
$c(\beta) = \frac{c_0}{2} \beta^{(\frac{-1}{n})}m_\phi^{1/n}$, and $C= C_0({\rm e}^{c_0}+1)$.
As $\beta>1$ can be chosen arbitrarily, we observe that since $\phi$ and $\varphi$ are assumed to be equisingular, the function $\exp \big( c(\beta) E_{\chi_1}(\varphi, \phi)^{-\frac{1}{n}} \chi_2(\phi-\varphi)\big)$ is uniformly bounded, so by Lebesgue dominated convergence theorem, we can choose \smash{$c = \frac{c_0}{2} m_\phi^{1/n}$}.
Observe that the constants $c$ and $C$ are independent of $\chi_1$ (and $\chi_2$).
\end{proof}

\begin{corollary}
Let $\phi$ be a model potential with $\int_X \theta_\phi^n>0$. Let $\chi_1\colon\R^+\rightarrow \R^+ $ be a weight. Let \smash{$\chi_2(t) := \int_0^t \chi_1(s)^{\frac{1}{n}} {\rm d}s$}. Let $\varphi\in \mathcal{E}_{\chi_1}(X,\theta, \phi)$ with $\sup_X \varphi=-1$. For $t \geq 0 $, we define
\[\Omega(t) := \{ x \in X\colon (\phi-\varphi)(x) \geq t \geq 0 \},\]
 and let $m(t)$ be the mass of $\Omega(t)$ with respect to $\omega^n$. Then there exists $ S > 1$ depending only on~$X$,~$\theta$, $n$, $\omega$ such that
\[ \frac{ m_\phi \big(\int_0^{+\infty} m(t) \chi_1(t)^{1/n} {\rm d}t\big)^n}{S} \leq E_{\chi_1}( \varphi, \phi).\]
\end{corollary}
\begin{proof}
Observe that again all inequalities are unchanged if we replace $\chi_1$ with $\alpha \chi_1$ where $\alpha$ is a positive constant, so we assume $\chi_1(1) = 1$.
Let $\Omega = \{(x,t) \in X \times \mathbb{R}^+ \colon (\phi - \varphi)(x) \geq t\geq 0\}$.
Since $\omega$ has volume $1$, by Jensen's inequality and by Theorem \ref{thm: rel Moser--Trudinger inequality weight}, we find
\[c E_{\chi_1}(\varphi,\phi)^{-\frac{1}{n}} \int_X \int_0^{(\phi-\varphi)(x)} \chi_1(t)^{1/n} {\rm d}t \ \omega^n = c E_{\chi_1}(\varphi, \phi)^{-\frac{1}{n}} \int_\Omega \chi_1(t)^{1/n} {\rm d}t \wedge \omega^n \leq \log C. \]
By the Fubini theorem, we obtain the inequality if we recall the value of the constants $C$ and $c$ given in the proof of \ref{thm: rel Moser--Trudinger inequality weight} and we let
\smash{$ S = \big(\frac{2 \log (C_0({\rm e}^{C_0} + 1))}{c_0 }\big)^n$}.

Also, observe that, if we choose $\varphi = \phi-1$ and $\chi_1(t) = t^p$ with $p > 0 $, we find \smash{$ \frac{S}{m_\phi} \geq \frac{n^n}{(p+n)^n m_\phi}$}. If we let $p$ go to zero, we conclude that $\frac{S}{m_{\phi}} \geq \frac{1}{m_{\phi}}$, i.e., $S \geq 1$.

If we apply the inequality choosing $\varphi$ as in the statement of the corollary with $\varphi \neq \phi -1$, we see that in fact $S > 1$.
\end{proof}

\begin{Remark}
\label{rem:Case p}
Observe that $\chi_1(t)=t^p$ and $\chi_2(t)= q^{-1}t^q$ with $q=1+p/n$ satisfy the assumptions of Theorem~\ref{thm: rel Moser--Trudinger inequality weight}. With this particular choice of weights, we have that
\[\gamma_2(t)=(a^{-1}q)^{1/q} t^{1/q}, \qquad \gamma_2'(t)=a^{-1/q} q^{\frac{1-q}{q}} t^{\frac{1-q}{q}}\] and
 \[ \tau_2(t)= a^{\frac{1}{1-q}} q^{-1}t^{\frac{q}{1-q}}=a^{-\frac{n}{p}} q^{-1}t^{-\frac{n+p}{p}}.\]
 Thus, in this particular case $\tau_2(1)=a^{-\frac{n}{p}} q^{-1} $, that is, (up to a power of $q$) the constant appearing in \cite[Theorem 2.11]{DNL21}.
\end{Remark}

As consequence of Theorem \ref{thm: rel Moser--Trudinger inequality weight}, we then get:

\begin{Theorem}	\label{thm: Entropy implies rel Ep}
Let $\phi$ be a model potential with $\int_X \theta_\phi^n>0$. Assume that $n > 1$,	
 fix $B>0$ and set $p=\frac{n}{n-1}$.
	There exist $c,C>0$ depending only on $B$, $X$, $\theta$, $\omega$, $n$ such that for all $\varphi\in \Ent_B(X,\theta) \cap\mathcal{E}(X,\theta,\phi)$, we have
\[
	\int_X {\rm e}^{c (\phi-\varphi)^p} \omega^n \leq C \qquad \text{and} \qquad E_p(\varphi, \phi) \leq C.
\]
	In particular, $\Ent(X,\theta) \cap\mathcal{E}(X,\theta,\phi) \subset \mathcal{E}^{\frac{n}{n-1}}(X,\theta, \phi)$.

If $n=1$ and $ \varphi\in \Ent(X,\theta) \cap\mathcal{E}(X,\theta,\phi)$, then $\varphi$ has the same singularity type of $\phi$. In particular, $\varphi\in \mathcal{E}^{p}(X,\theta, \phi)$ for any $p\geq 1$.
 \end{Theorem}

 Here
\[
\Ent_B(X,\theta) := \bigl\{u \in \psh(X, \theta)\colon \sup_X u =-1, \ \Ent_\theta(u) \leq B \bigr\}.
\]

We also emphasize the constants $c, C>0$ in the statement do not depend on $\phi$.

\begin{proof}
 We consider the convex function $\chi\colon s \in \R^+ \mapsto (s+1) \log (s+1) -s \in \R^+$. Its conjugate convex function is
\[
\chi^*\colon\ t \in \R^+ \mapsto \sup_{s>0} \{ st-\chi(s) \}={\rm e}^t-t-1 \in \R^+.
\]

By definition, these functions satisfy, for all $s,t>0$,
\begin{equation} \label{eq:conj}
st \leq \chi(s)+\chi^*(t).
\end{equation}

Set $\mu=\theta_\varphi^n = f \omega^n$.	We claim that for any $v\in \mathcal{E}^p(X,\theta, \phi)$ we have $\int_X |v-\phi|^p {\rm d}\mu <+\infty$. Indeed, fix $v \in \mathcal{E}^p(X,\theta, \phi)$ with $\sup_X v=-1$ and observe that $1+\frac{p}{n} = p$. It follows from Theorem \ref{thm: rel Moser--Trudinger inequality weight} (see also Remark \ref{rem:Case p}) that for some $c>0$ small enough
\begin{equation}\label{unif_bound_exp}
\int_X {\rm e}^{c (\phi-v)^p} \omega^n <+\infty.
\end{equation}

We apply \eqref{eq:conj} with $s=f(x)$ and $t=c(\phi(x)-v(x))^p$. This yields
\[
\int_X c(\phi-v)^p \theta_\f^n= \int_X c (\phi-v)^p f \omega^n \leq \int_X \chi\circ f \omega^n +\int_X ({\rm e}^{c (\phi-v)^p}-c (\phi-v)^p-1) \omega^n,
\]
where the first integral is finite since $\varphi$ has finite entropy, while the second is bounded thanks to~\eqref{unif_bound_exp} and the integrability properties of qpsh functions.
This means that $\int_X (\phi-v)^p \theta_\f^n<+\infty$, proving the claim. By \cite[Theorem~1.4]{DDL6}, we can thus infer that $\f\in \mathcal{E}^p(X,\theta, \phi)$.

Using \eqref{eq:conj} again, we see that
 \begin{align*}
 	 \int_X c \frac{(\phi-\f)^{p}}{ E_p(\f, \phi)^{1/n}} f \omega^n
 	 & \leq \int_X \left({\rm e}^{c \frac{(\phi-\f)^{p}}{ E_p(\f, \phi)^{1/n}}} -c \frac{(\phi-\f)^{p}}{ E_p(\f, \phi)^{1/n}}-1 \right)\omega^n
 	 + \int_X \chi\circ f \omega^n.
 \end{align*}
 The right-hand side is uniformly bounded by $C_1>0$ thanks to Theorem \ref{thm: rel Moser--Trudinger inequality weight} and the fact that~${\Ent_\theta(\varphi, \phi)< B}$. From the above inequality and the fact that $\theta_\varphi^n= f\omega^n$, we get \[c { E_p(\f, \phi)^{-1/n}} \int_X {(\phi-\f)^{p}} \theta_\varphi^n = c E_p(\f, \phi )^{1-1/n} \leq C_1,\]
 which yields $E_p(\f, \phi) \leq C_2$.
Finally, invoking Theorem \ref{thm: rel Moser--Trudinger inequality weight} again, we obtain{\samepage
\[
\int_X {\rm e}^{\gamma (\phi-\f)^p} \omega^n \leq C_3,
\]
 where \smash{$\gamma=c C_2^{-1/n}>0$} is a uniform constant.}

We now treat the case of Riemann surfaces.

Assume $\varphi\in \Ent(X,\theta)\cap \mathcal{E}(X,\theta,\phi)$. We first proceed similarly to \cite[Lemma~3.16]{Trus20b} to prove that $\SH(X,\omega)\subset L^1(\mu)$ for $\mu=\theta_\varphi$. The conclusion will then follow from the lemma below. We can assume without loss of generality that $\int_X \omega=1$. Consider $u\in \SH(X,\omega)$, $u\leq 0$ and let $c>0$ such that~${{\rm e}^{-cu}\in L^1(\omega)}$. Setting $a:=\int_X {\rm e}^{-cu}\omega$, we have
\[
0\leq \Ent\big(m_\varphi^{-1}\mu, a^{-1}{\rm e}^{-cu}\omega\big)=\Ent\big(m_\varphi^{-1}\mu, \omega\big)+\log a+c\int_X u {\rm d}\mu\leq C+ \log a+c\int_X u {\rm d}\mu,
\]
which yields $u\in L^1(\mu)$ and concludes the claim.
 \end{proof}

\begin{Lemma}
Assume $n =1$ and $\varphi\in \mathcal{E}(X, \theta, \phi)$. Then $|\varphi-\phi|< C$, $C>0$ if and only if~${\SH(X, \omega)\subset L^1(\theta_\varphi)}$.
\end{Lemma}
\begin{proof}
Since there could be confusion, in the proof we use the notation $\langle \theta_\varphi\rangle$ to denote the non-pluripolar part of the measure $\theta_\varphi$. We want then to prove that $|\varphi-\phi|< C$, $C>0$ if and only if $\SH(X, \omega)\subset L^1(\langle \theta_\varphi\rangle)$. Observe that by definition $\langle \theta_\varphi\rangle = {\bf 1}_{\{\varphi>-\infty\}} \theta_\varphi$. Assume~${|\varphi-\phi|\leq C}$. Let $u\in \text{SH}(X, \omega) $ that without loss of generality can be assumed to be negative. Consider the bounded $\omega$-psh approximants $u_k:=\max(u, -k)$. Then using integration by parts \cite[Lemma~4.7]{DDL6} (with $\varphi_1=u_k$, $\varphi_2=0$, $\psi_1=\varphi$, $\psi_2=\phi$), we get
\begin{align*}
\int_X (-u_k) \langle \theta_{\varphi} \rangle
&= \int_{X} (-u_k) (\langle\theta_{\phi}\rangle + \langle\theta_{\varphi}\rangle -\langle\theta_{\phi}\rangle)\leq \int_{X} (-u_k) \langle \theta_\phi \rangle + \int_{X} (-u_k) (\langle\theta_{\varphi}\rangle -\langle\theta_{\phi}\rangle)\\
&\leq C_1 \int_X (-u_k) \omega +\int_X (\phi-\varphi) {\rm dd}^c u_k\\
&\leq C_1 \int_X (-u_k) \omega + C \int_X\omega_{u_k} - \int_X (\phi-\varphi)\omega\leq C_1 \int_X (-u_k) \omega + 2C \int_X\omega,
 \end{align*}
 where in the fifth line we used that $|\varphi-\phi|\leq C$, in the forth line that
\[\langle \theta_\phi \rangle\leq {\bf 1}_{\{\phi=0\}} \theta\leq {\bf 1}_{\{\phi=0\}} C_1 \omega\]
for some positive constant $C_1 $ (see \cite[Theorem 3.6]{DDL6}), while in the last line we used that ${\int_X \omega_{u_k} = \int_X \omega}$ and that $\varphi-\phi \leq C$. Since $\int_X (-u_k) \omega $ is uniformly bounded \cite[Proposition~1.7]{GZ05}, we arrive at
\[\int_X (-u_k) \langle \theta_{\varphi} \rangle \leq C_2\]
for some positive constant independent in $k$. Now, $\{-u_k\}$ is increasing to $-u$. Fatou's lemma ensures that
\[\int_X (-u) \langle \theta_{\varphi} \rangle \leq C_2,\]
that is, what we wanted to prove.

We now prove the viceversa. Let $\mu=\langle \theta_\varphi \rangle $. For any fixed $k\in \mathbb{N}$, we consider the following sequence of $\omega$-sh functions.
We choose a point $a \in X$. In a neighbourhood of $a$, which will be identified with the unit ball where the origin corresponds to $a$, we consider the bounded psh function $u_k(z):=\max(\log\|z\|, -k)$.
Such a psh function $u_k$ can be globalized to a genuine $\omega$-sh function, denoted by $\psi_a^k$, normalized with $\sup_X \psi_a^k=0$ (see, for example, \cite[p.~613]{GZ05}).
By assumption and \cite[Proposition~2.7]{GZ05},
\[\int_X \bigl(-\psi_a^k\bigr) {\rm d}\mu \leq C_{\mu}.\]
 Moreover, by construction we have that $\omega+{\rm dd}^c \psi_a^k$ is the Lebesgue measure $\sigma_k$ on the sphere $\|z|={\rm e}^{-k}$ (see \cite[Example~3.13]{GZbook}) normalized such that \smash{$\int_{\{\lVert z\rVert ={\rm e}^{-k}\}} {\rm d}\sigma_k=1$}.

Set $\varphi_j:=\max(\varphi, \phi-j)$. Observe that $\varphi_j \searrow \varphi$, and $ -j \leq \varphi_j-\phi\leq 0$. In particular, for any fixed $j$, $ \varphi_j-\phi$ is bounded.
By \cite[Remark~2.5]{DDL2}, we infer that
\[C_\mu \geq \int_X \bigl(-\psi_a^k\bigr) \langle\theta_\varphi \rangle= \lim_j \int_X \bigl(-\psi_a^k\bigr) \langle \theta_{\varphi_j}\rangle.\]

In particular, we can assume that there exists $j_0 \in \mathbb{N}$ ($j_0=j_0(k)$) big enough such that for~${j\geq j_0}$,
\[\int_X \psi_a^k \langle \theta_{\varphi_j}\rangle \geq -C_\mu -1.\]
Thus, for $j\geq j_0$, performing integration by parts \cite[Lemma~4.7]{DDL6} (applied with $\varphi_1=\varphi_j$, $\varphi_2=\phi$, $\psi_1= \psi_a^k$, $\psi_2=0$), we obtain
\begin{align*}
C_\mu +1 &\geq \int_X \bigl(-\psi_a^k\bigr) (\langle \theta_{\varphi_j}\rangle-\langle \theta_\phi \rangle)=\int_X (\varphi_j-\phi)\omega+\int_X (\phi-\varphi_j ) \omega_{\psi_a^k}\\
&\geq \int_X \varphi_j \omega+ \int_{\{\lVert z\rVert ={\rm e}^{-k}\}} (\phi-\varphi_j ) {\rm d}\sigma_k\geq -C_0+ \frac{1}{2\pi}\int_{0}^{2\pi} (\phi-\varphi_j )\big({\rm e}^{{\rm i}\theta} {\rm e}^{-k}\big) {\rm d}\theta,
\end{align*}
where in the last inequality we used that $\int_X (-\varphi_j) \omega \leq \int_X (-\varphi) \omega \leq C_0$, for some $C_0>0$ uniform in $j$.
As $j\rightarrow +\infty$, monotone convergence gives
\[ C_\mu +C_0 +1 + \frac{1}{2\pi }\int_{0}^{2\pi} \varphi ({\rm e}^{{\rm i}\theta} {\rm e}^{-k}) {\rm d}\theta \geq \frac{1}{2\pi }\int_{0}^{2\pi} \phi({\rm e}^{{\rm i}\theta} {\rm e}^{-k}) {\rm d}\theta.\]
Then as $k\to +\infty$, thanks to \cite[Proposition 1.13]{GZbook}, we get
\[ C_\mu +C_0 +1 +\varphi(a) \geq \phi (a),\]
hence the conclusion.
\end{proof}

\subsection{Stability of the entropy}
We collect some results on how the property of having finite entropy changes when the reference probability measure changes, when we perturb the reference big cohomology class and when we pull-back through bimeromorphic holomorphic maps.

 \begin{Lemma}
 \label{lem:Sub_Ent}
 Let $\mu_1$, $\mu_2$, $\mu_3$ be probability measures on $X$. If $\mu_2=f_2 \mu_3$ with $f_2\in L^\infty(X)$, then
 \begin{gather} \label{eqn:Sub_Ent}
 \Ent(\mu_1,\mu_3)\leq\Ent(\mu_1,\mu_2)+\log\bigl(\sup_{X}f_2\bigr).
 \end{gather}
 \end{Lemma}
 \begin{proof}
If $\Ent(\mu_1,\mu_2)=+\infty$ the above inequality is trivial, so we can assume that $\Ent(\mu_1,\mu_2)<+\infty$. In particular, $\mu_1= f_1\mu_2$ with $f_1\geq 0$ such that $\int_X f_1 \log f_1 {\rm d}\mu_2<+\infty$. By assumption, we have $\mu_1=f_1f_2 \mu_3$. In particular, $f_1f_2\in L^1(\mu_3)$ and
\[ \int_X f_1 \log f_1 {\rm d}\mu_2= \int_X f_1 f_2 \log f_1 {\rm d}\mu_3 <+\infty.
 \]
 We then observe that
 \[
 (f_1f_2) \log(f_1f_2)\leq (f_1 f_2)\log f_1 + (f_1f_2)\log \bigl(\sup_X f_2\bigr)
 \]
 giving that
 \[\int_X f_1 f_2 \log (f_1 f_2) {\rm d}\mu_3 \leq \int_X f_1 \log f_1 {\rm d}\mu_2+ \log \bigl(\sup_X f_2\bigr) \int_X f_1 \ {\rm d}\mu_2. \]
 From this, we deduce \eqref{eqn:Sub_Ent}.
 \end{proof}

 \begin{Proposition}
 \label{lem:Holo}
 Let $\pi\colon Y\to X$ be a bimeromorphic and holomorphic map and assume that~$\tilde{\omega}$ is K\"ahler form on $Y$ normalized with volume equal to $1$. Then
 \begin{itemize}
 \item[$(i)$] $\Ent (\mu, \nu)=\Ent (\pi^\star \mu, \pi^\star \nu) $, for any two non-pluripolar probability measures $\mu$, $\nu$.
 \item[$(ii)$] If $\Ent_\theta(\varphi) <+\infty$, then $\Ent\big(m_\varphi^{-1} \pi^\star \theta^n_\varphi , \tilde{\omega}^n\big)<+\infty$. In particular, $\Ent_{\pi^\star \theta} (\pi^\star \varphi)<+\infty$.
 \end{itemize}
 \end{Proposition}
By $\pi^*\mu$, we mean the pushforward by $\pi^{-1}$ of $\mathbf{1}_{X\setminus Z}\mu$ where $Z$ is the indeterminacy locus of~$\pi^{-1}$ (see also \cite[lines after Definition~1.3]{BBEGZ19}).
 \begin{proof}
 For the first item, we simply note that if $\mu=f\nu$ then $\pi^\star \mu= (f\circ \pi) \pi^\star \nu$. The conclusion follows since $\pi$ is a biholomorphism on a Zariski dense open subset and $\mu$, $\nu$ do not charge pluripolar sets.

 Since $\Ent_\theta(\varphi)<+\infty$, we have $\theta_\varphi^n=f\omega^n$. We recall that \cite[Theorem~A]{DN14} ensures that $(\pi^*\theta_\varphi)^n= \pi^*\theta^n_\varphi$. The first item then imply
 \[
 \Ent\big(m_\varphi^{-1} (\pi^*\theta_\varphi)^n ,\pi^*\omega^n\big)<+\infty.
\]
 Also, as $\pi\colon Y\to X$ is holomorphic, we know that $\pi^*\omega$ is a semipositive smooth form. In particular, $\pi^*\omega^n=F \tilde{\omega}^n$ for some $F\in L^\infty(Y)$, $F\geq 0$. Lemma \ref{lem:Sub_Ent} (with $\mu_1=m_\varphi^{-1}\pi^*\theta^n_\varphi$, $\mu_2=\pi^*\omega^n$ , $\mu_3=\tilde{\omega}^n$ and $f_2=F$) gives
 \[
 \Ent(\mu_1, \mu_3 )<+\infty.\tag*{\qed}
 \] \renewcommand{\qed}{}
 \end{proof}

We end this note with a very natural stability question.

\begin{Question}
Let $\varphi \in \Ent(X, \theta)$. Is true that $\varphi \in \Ent(X, \theta+\varepsilon \omega)$ for $\varepsilon>0$?
\end{Question}

\begin{Proposition}\label{lemma: entropy all time}
Assume $\varphi \in \Ent(X, \theta)$ and that $\varphi \in \Ent(X, \theta+\varepsilon \omega)$ for some $\varepsilon>0$. Then~${\varphi \in \Ent(X, \theta+t\omega)}$ for all $t>0$.
\end{Proposition}

\begin{proof}
Set $L\colon \R^+\rightarrow \R$ such that $L(x)= x \log x$, let $\overline{L}:=\max(L,0)$. Then the function $\overline{L}$ is convex increasing and non-negative. If $f$ is a non-negative measurable function, since $L$ is bounded on $[0,1]$,
\[\int_X L(f) \omega^n < +\infty \qquad \text{if and only if}\qquad \int_X \overline{L}(f) \omega^n < +\infty. \]

We know that $\theta_u^n = f \omega^n$, $(\theta_u+\epsilon\omega)^n=(\theta+\epsilon\omega +{\rm dd}^c u)^n = h \omega^n$ with
\[\int_X f \log f< +\infty, \qquad \int_X h\log h< +\infty.\]
On the other hand, from the multilinearity of the non-pluripolar product it follows that
\[(\theta_u+\varepsilon\omega)^n =\theta_u^n + \sum_{j=0}^{n-1} \binom{n}{j} \varepsilon^{n-j} \omega^{n-j} \wedge \theta_u^j=\theta_u^n + \sum_{j=0}^{n-1} \binom{n}{j} \varepsilon^{n-j}\mu_j,\]
where $\mu_j:=\omega^{n-j}\wedge \theta_u^j$ are positive measures. Therefore, since the Monge--Amp\`ere measures~${(\theta_u\!+\varepsilon\omega)^n}$ and $\theta_u^n$ are absolutely continuous with respect to $\omega^n$, we have that if ${\omega^n(E)=0}$, $E\subset X$, then $\mu_j(E)=0$. This means that $\mu_j = h_j \omega^n$ for some $0\leq h_j\in L^1(\omega^n)$. Clearly, $h_0=1$.

Now, for any $t>0$, we have
\begin{equation*}
 (\theta_u+t\omega)^n= \theta_u^n+ \sum_{j=1}^{n-1} \binom{n}{j} t^{n-j}\mu_j+t^n \omega^n=\left( f +\sum_{j=1}^{n-1} \binom{n}{j} t^{n-j} h_j +t^n\right) \omega^n.
\end{equation*}
In particular, $(\theta_u+t \omega)^n$ is absolutely continuous with respect to $\omega^n$. Set
 \[S(t):=f+\sum_{j=1}^{n-1} \binom{n}{j} t^{n-j} h_j +t^n.\]

Note that $S(t)$ is non negative and increasing in $t$.
 We want to prove that the density $S(t)$ is such that $\int_X \overline{L}(S(t)) < +\infty $.

If $0 < t \leq \varepsilon$, then $\overline{L}(S(t)) \leq \overline{L}(S(\varepsilon))$. If $t = \tau \varepsilon$ with $\tau \geq 1$, then
\[
0\leq S(t)= S(\tau \varepsilon) \leq \tau^n S(\varepsilon).\]
Then
\begin{align*}
0 \leq \overline{L}(S(t)) &\leq \overline{L}(\tau^n S(\varepsilon)) ={\bf 1}_{\{\tau^n S(\varepsilon) \geq 1\}}( \tau^n \log \tau^n S(\varepsilon) ) + {\bf 1}_{\{\tau^n S(\varepsilon) \geq 1\}} ( \tau^n L (S(\varepsilon))) \\
&\leq \tau^n \log \tau^n S(\varepsilon) + \tau^n \overline{L}(S(\varepsilon)).
 \end{align*}
The quantity of the right-hand side is finite by assumption. Hence the conclusion.
 \end{proof}

\subsection*{Acknowledgements}

The first author is supported by the project SiGMA ANR-22-ERCS-0004-02 and by the ANR-21-CE40-0011 JCJC MARGE. The second author is partially supported by PRIN \emph{Real and Complex Ma\-ni\-folds: Topology, Geometry and Holomorphic Dynamics} no.~2017JZ2SW5, and by MIUR Excellence Department Projects awarded to the Department of Mathematics, University of Rome Tor Vergata, 2018--2022 CUP E83C18000100006, and 2023--2027 CUP E83C23000330006. The third author is supported by the ``Knut and Alice Wallenberg Foundation''.
 The authors thank Chinh Lu for discussions and the anonymous referee for his comments which helped improving the paper.

\pdfbookmark[1]{References}{ref}
\LastPageEnding

\end{document}